\documentclass[12pt]{amsart}
\usepackage{amssymb,amsmath}
\usepackage{geometry}    
\usepackage{mathrsfs}

\usepackage{graphicx}

\usepackage{vmargin}

\usepackage{tikz}  

\setmargrb{1in}{1in}{1in}{1in}

\newtheorem{theorem}{Theorem}[section]
\newtheorem{lemma}[theorem]{Lemma}
\newtheorem{proposition}{Proposition}
\newtheorem{corollary}[theorem]{Corollary}
\theoremstyle{definition}
\newtheorem{definition}{Definition}
\newtheorem{remark}{Remark}

\newtheorem{conjecture}{Conjecture}
\newtheorem{problem}{Problem}

\usepackage{yfonts}

\usepackage{xcolor}

\begin{document}
	
	\title[Mahler's  Conjecture on Liouville numbers fails for matrices]{ Maillet's property and Mahler's  Conjecture on Liouville numbers fail for matrices }
	
	\author{Johannes Schleischitz}

	\thanks{Middle East Technical University, Northern Cyprus Campus, Kalkanli, G\"uzelyurt \\
		johannes@metu.edu.tr ; jschleischitz@outlook.com}

\begin{abstract}
      In the early 1900's, Maillet proved
    that the image of any Liouville number under a rational function with rational coefficients is again a Liouville number. The analogous result
    for quadratic Liouville matrices in higher dimension turns out to fail.
     In fact, using a result by Kleinbock and Margulis, we
     show that among analytic matrix functions in
     dimension $n\ge 2$,
     Maillet's invariance property is only true for M\"obius transformations with special coefficients. This implies that the analogue in higher dimension of an open question of Mahler on the existence of transcendental entire functions with Maillet's property has a negative answer.
     On the other hand, extending a topological argument of 
     Erd\H{o}s, we prove that for any injective continuous self mapping on the space of rectangular matrices, many Liouville matrices are mapped to Liouville matrices.
     Dropping injectivity, we consider setups similar to Alnia\c{c}ik and Saias and show that the situation depends on the matrix dimensions $m,n$. Finally we discuss extensions of a related result by Burger to quadratic matrices.
     We state several open problems along the way.
\end{abstract}

\maketitle

{\footnotesize{

		Math Subject Classification 2020: 30B10, 11J13}}

\section{ Introduction: Maillet's property and Mahler's problem }  \label{intro}

A Liouville number is an irrational real number $x$ for which
$|x-p/q|<q^{-N}$ has a rational solution $p/q$ for any $N$. 
Denote by $\mathscr{L}=\mathscr{L}_{1,1}$ the set of Liouville numbers.
It is well-known that
$\mathscr{L}$ is comeagre,
equivalently a dense $G_{\delta}$ set,
and of Hausdorff dimension $0$,
see~\cite{ox}.
Maillet~\cite{maillet} proved that if $f$ is a non-constant rational function with rational
coefficients then 
$f(a)\in \mathscr{L}$ for any $a\in \mathscr{L}$, or equivalently $f(\mathscr{L})\subseteq \mathscr{L}$. We will say a real function $f: I\subseteq \mathbb{R}\to \mathbb{R}$, $I$ a non-empty open interval, has the Maillet property if $f(I\cap \mathscr{L})\subseteq \mathscr{L}$ or equivalently $f^{-1}(\mathscr{L})=f^{-1}(\mathscr{L})\cap \mathscr{L}= \mathscr{L}\cap I$.
An open question
by Mahler~\cite{mahler} asks 
the following in the classical setup $I=\mathbb{R}$:

\begin{problem} \label{111}
    Do there exist 
transcendental entire functions having Maillet's property?
\end{problem}

 Recall a function $f$ is called
transcendental if $P(z,f(z))\ne 0$ for any non-trivial bivariate polynomial 
$P(X,Y)= \sum c_{i,j} X^i Y^j$ with complex coefficients $c_{i,j}$, otherwise
$f$ is called algebraic. Some advances
to Mahler's question by providing entire transcendetal functions $f$ mapping large subclasses
of Liouville numbers into 
$\mathscr{L}$ (or even itself)
were made in~\cite{m2, ds}, see
also~\cite{s}.
Conversely, claims mildly indicating towards a negative
answer of Problem~\ref{111} in context
of~\cite[Corollary~2.2]{ds},
were obtained in~\cite{LM, MRS, MS}.
Besides Maillet's result, it is known that any ``reasonable'' function enjoys the weaker property that, while not all, many Liouville numbers are mapped to Liouville
numbers. Indeed, as shown in~\cite{asa} 
for any continuous function $f$ as above that is nowhere constant, the set $f^{-1}(\mathscr{L})\cap \mathscr{L}$ is a dense $G_{\delta}$ subset of Liouville numbers on $I$. In fact we can intersect over countably many preimages under such functions $f_k$ at once. See also the preceding papers~\cite{erd, rie, sch}.

In this paper we want to discuss analogous
claims for Liouville matrices to be defined. 
Let $\Vert.\Vert$ denote the supremum norm on a Euclidean space of any dimension. 

\begin{definition} \label{lio}
We call a real $m\times n$ matrix $A$ Liouville matrix if 
\begin{equation} \label{eq:tiere}
    A\cdot \textbf{q}-\textbf{p} \ne \textbf{0}, \qquad (\textbf{q},\textbf{p})\in\mathbb{Z}^{n+m}\setminus \{ \textbf{0}\},
\end{equation}
 and $\Vert A\textbf{q}-\textbf{p}\Vert< \Vert \textbf{q}\Vert^{-N}$ has a solution in integer
vectors $\textbf{p}\in\mathbb{Z}^m,\textbf{q}\in\mathbb{Z}^n\setminus \{ \textbf{0}\}$ for any $N$. We denote 
by $\mathscr{L}_{m,n}$ the set 
of $m\times n$ Liouville matrices.
\end{definition}

This agrees
with the definition of Liouville numbers if $m=n=1$, so $\mathscr{L}=\mathscr{L}_{1,1}$. 
Studying small values of $\Vert A\textbf{q}-\textbf{p}\Vert$ is an intensely studied topic in Diophantine approximation, e.g.~\cite{ngm},
so the definition appears very natural.
Property \eqref{eq:tiere} means that the sequence of best approximating integer vectors associated to $A$ does not terminate (called good matrices 
in~\cite{ngm}), we consider
it more natural than the less restrictive condition
$A\notin \mathbb{Q}^{m\times n}$, see however Theorem~\ref{neu} on the latter. Since approximation becomes easier
the more free variables we have 
and the fewer equations we need to satisfy, the following observations are obvious but may be helpful for the reader.

\begin{proposition}  \label{hauser}
    If $A\in \mathscr{L}_{m,n}$ then any line $(a_{j,1},\ldots,a_{j,n})$ of $A$ either has $\mathbb{Z}$-linearly dependent coordinates together with $\{1\}$ or lies in $\mathscr{L}_{1,n}$. If some column of
    $A$ lies in $\mathscr{L}_{m,1}$
    and $A$ satisfies \eqref{eq:tiere},
    then $A\in \mathscr{L}_{m,n}$.
\end{proposition}

 In general
$\mathscr{L}_{n,m}\ne \mathscr{L}_{m,n}^{t}$, superscript $t$ denoting the transpose, moreover $A\in\mathscr{L}_{n,m}$ by no means implies that its entries are Liouville numbers, nor is the converse true.

\section{ Mahler's question has a negative answer for matrices} \label{s02}

We focus on $m=n$ in this section.\footnote{In this
case we may additionally impose for Liouville matrices the constraint of being transcendental, that is imposing $P(A)\ne \textbf{0}$ for any non-zero $P\in\mathbb{Z}[X]$. 
This would exclude especially nilpotent matrices (it is not hard to construct nilpotent Liouville matrices in our setting 
for $n\ge 2$). 
Our
results below remain valid in this setting as well with minor modifications in some proofs. Another reasonable restriction would be to only consider invertible matrices. }
Given $I\ni 0$ an open interval
and any analytic function $f: I\to \mathbb{R}$, we extend it
to (a non-empty open connected subset of) the ring of $n\times n$ matrices via the same local power series expansion. More precisely, we know that
$f(z)= \sum c_j z^j$ converges
absolutely in some subinterval $(-r,r)\subseteq I$, $r>0$, and for $A\in \mathbb{R}^{n\times n}$ we denote by $\textswab{f}$ the extension of $f$ to the matrix ring via
\begin{equation}  \label{eq:ex}
    \textswab{f}(A)= \sum_{j=0}^{\infty} c_j A^j,
\end{equation}
which converges absolutely as soon as $A$ has operator norm less than $r$. Denoting $I_n$ the identity matrix,
this setup appears natural
and contains for example 
any rational function $(c_0 I_n+c_1 A+\ldots+c_uA^u)\cdot(d_0 I_n+\cdots+d_v A^v)^{-1}$, $d_0\ne 0$,
the matrix
exponential function $e^A=\sum_{j\ge 0} A^j/j!$ and 
its inverse $\log (I_n+A)= \sum_{j\ge 1} (-1)^{j+1} A^{j}/j$.

The following is not hard to see, we provide a sketch of the proof in~\S~\ref{unten}.

\begin{proposition}  \label{pro}
    If $A\in \mathscr{L}_{n,n}$ 
    and 
    $R_1, R_2, S, T\in \mathbb{Q}^{n\times n}$
    and $R_i$ are invertible, then
    \begin{equation}    \label{eq:fini}  
    R_1AR_2+S, \qquad (R_1AR_2+T)^{-1}+ S
    \end{equation}
    again belong to $\mathscr{L}_{n,n}$ (if defined in the latter case).
\end{proposition}

\begin{remark}
    The regularity of $R_i$ is only needed
    to avoid rational matrices in \eqref{eq:fini} thereby not satisfying \eqref{eq:tiere}, especially if $A$ is not invertible.
    A refined claim on invariance of the irrationality exponent defined below can be obtained. Moreover some claims can be generalized to rectangular matrices.
\end{remark}

Let us extend naturally Maillet's property for real $f$ 
to given $n\ge 1$ via imposing
$\textswab{f}(\mathscr{L}_{n,n})\subseteq \mathscr{L}_{n,n}$ for the induced $\textswab{f}$ from \eqref{eq:ex}.
Taking in \eqref{eq:fini} diagonal matrices
\[
R_1=rI_n,\; (r\ne 0),\qquad  R_2=I_n,\qquad S=sI_n,\qquad T=tI_n,
\]
Maillet's property holds for any $n\ge 1$ and the two types of algebraic functions
\begin{equation} \label{eq:fro}
f(z)= r z+s, \qquad f(z)=s+ (rz+t)^{-1}
\end{equation}
with $0\ne r,s,t\in\mathbb{Q}$. 
Equivalently, $f$ is a M\"obius map
with rational coefficients
\begin{equation} \label{eq:four}
    f(z)=\tau_{a,b,c,d}(z):= \frac{ az+b}{cz+d}, \qquad a,b,c,d\in\mathbb{Q},\; ad-bc\ne 0.
\end{equation}

However,
we show that 
for $n\ge 2$, Maillet's property fails
for any other analytic function not of the form \eqref{eq:fro} with real parameters. Thereby we get
an almost comprehensive description of analytic functions with Maillet's property for $n\ge 2$, leaving  only a small gap  of the case \eqref{eq:fro} 
(or \eqref{eq:four}) with non-rational constants.
In fact we show
a refined claim that needs some more preparation. For general rectangular matrices $A\in \mathbb{R}^{m\times n}$, let $\omega^{m\times n}(A)$ be the irrationality 
exponent of $A$,
defined as the supremum of $w$ such that
\[
\Vert A \textbf{q}-\textbf{p}\Vert < \Vert \textbf{q}\Vert^{-w}
\]
for infinitely many integer vector pairs $\textbf{p}\in \mathbb{Z}^m, \textbf{q}\in \mathbb{Z}^n$. We should notice that
\[
\omega^{m\times n}(A)\ge \frac{n}{m}
\]
for any real $A$ by a well-known variant Dirichlet's Theorem, in particular for 
quadratic matrices the lower bound is $1$. 
Let us just write $\omega(A)$ when the dimensions are clear.
A real $m\times n$ matrix $A$ is a Liouville matrix if and only if \eqref{eq:tiere} holds and $\omega(A)=\infty$. In the sequel we shall always naturally identify
\[
\mathbb{R}^{m\times n}\cong \mathbb{R}^{mn},
\]
where we can assume that the
lines of the matrix are put one by one into a vector (we can choose any coordinate ordering, however we must stick to it as $\mathscr{L}_{m,n}$ is not invariant under entry bijections
as soon as $\min\{ m,n\}\ge 2$).
This induces a topology and a Lebesgue measure (generally Hausdorff measures) on the matrix set.
Our main result is

\begin{theorem}  \label{1}
    Let $n\ge 2$ be an integer and $I\subseteq \mathbb{R}$ an open interval containing $0$. 
    Let $f_k: I\to \mathbb{R}$ be any sequence of real analytic functions
    not of the form \eqref{eq:fro} for real numbers $r, s, t$ (possibly $0$) and $\textswab{f}_k$ their extensions as in \eqref{eq:ex}.
    Then there exists a set 
    $\Omega\subseteq \mathscr{L}_{n,n}\subseteq \mathbb{R}^{n^2}$ of Hausdorff dimension
 $\dim_H(\Omega)=(n-2)^2+1$ so that
 for any $A\in \Omega$ and $k\ge 1$, 
	$\textswab{f}_k(A)$ is defined but
 $\textswab{f}_k(A)\notin \mathscr{L}_{n,n}$, thus $\cup_{k\ge 1} \textswab{f}_k^{-1}(\mathscr{L}_{n,n})\cap \Omega=\emptyset$. 
 In short
 \[
 \dim_H\left( \bigcap_{k\ge 1} \textswab{f}_k^{-1}(\mathscr{L}_{n,n}^c)\cap \mathscr{L}_{n,n}\right)\ge (n-2)^2+1.
 \]
 In fact for any $A\in \Omega$
	\[
	\omega(\textswab{f}_k(A)) \le 2, \qquad k\ge 1.
	\]
 If $f_k(0)=0$ for some $k$, then equality $\omega(\textswab{f}_k(A)) = 2$ can be obtained.
\end{theorem}

\begin{remark}
Note that we exclude all real $r,s,t$, not only rationals. We believe 
that when choosing any $r,s, t$ within the uncountable subset 
of $\mathscr{L}_{1,1}$
of so-called 
strong Liouville numbers,
the functions $\textswab{f}$ derived from $f$ in \eqref{eq:fro} satisfy $\textswab{f}(\mathscr{L}_{n,n})\subseteq \mathscr{L}_{n,n}$. By a result of Petruska~\cite{petruska} this is true
    for $n=1$, indeed if
    $c_0, c_1$ are strong Liouville numbers then e.g.
    $f(a)=c_0+c_1 a$
    is a Liouville number for any Liouville number $a$. In fact
    a weaker property of $c_i$ being so-called semi-strong
    Liouville numbers~\cite{al2} suffices.
    The general case $n\ge 1$ seems to admit a similar proof of $\textswab{f}(A)=c_0 I_n+c_1A\in \mathscr{L}_{n,n}$ for any $A\in \mathscr{L}_{n,n}$.
\end{remark}

\begin{remark} We believe that 
$\omega(\textswab{f}_k(A))=1$ can be 
reached as well in the framework of the theorem, possibly even badly approximable, i.e. 
$\Vert \textswab{f}_k(A)\textbf{q}-\textbf{p}\Vert\ge c_k \Vert \textbf{q}\Vert^{-1}$ for any integer vectors $\textbf{p},\textbf{q}\ne \textbf{0}$ and absolute $c_k>0$ (maybe even with uniform $c_k=c$).
However even considering only one function this would require new ideas in the proof.
\end{remark}

There is no reason to believe that the stated lower bound on $\dim_H(\Omega)$ is sharp, see also~\S~\ref{s4} below. On the other hand, topologically $\Omega$ is small. Indeed, it follows from Theorem~\ref{gdel} below
that even for only one function $f=f_1$ it must be meagre.
For our proof of Theorem~\ref{1}, it will be convenient to
use a deep result by Kleinbock and Margulis~\cite{km}. However, weaker partial claims 
can be obtained with elementary 
methods.

Since the functions in \eqref{eq:fro}
(or \eqref{eq:four}) are algebraic, Theorem~\ref{1} also shows that
Mahler's Problem~\ref{111} has a negative answer in higher dimension.

\begin{corollary} \label{c}
Let $n\ge 2$. For any transcendental entire function $f$,
its extension $\textswab{f}$ to the $n\times n$ matrix ring via \eqref{eq:ex} 
does not have the Maillet property.
\end{corollary}

It is however not clear if Corollary~\ref{c} can be regarded a strong indication for a negative answer in  the case $n=1$ (Mahler's Problem~\ref{111}) as well. Indeed,
the following remarks illustrate that the situation over the matrix ring is different from the scalar one.

\begin{remark} \label{remrand}
The proof of Theorem~\ref{1} shows that we can choose the matrix in the corollary an ultra-Liouville matrix
defined analogously to~\cite{m2} for $n=1$,
so the main claim from~\cite{m2} of the invariance of this set under certain entire transcendental maps
fails for $n\ge 2$ as well. The result \cite[Theorem~4.3]{ds} stating that the parametrized subclasses of 
Liouville numbers from~\cite[Definition 3.1]{ds} are mapped to Liouville numbers for certain transcendental functions $f$, fails for $n\ge 2$ as well by similar arguments.
\end{remark}

\begin{remark}
    In~\cite[Corollary~2.2]{ds} it is shown that if an entire function $f:\mathbb{R}\to\mathbb{R}$ satisfies $f(\mathbb{Q})\subseteq \mathbb{Q}$ with denominators of order $\rm{denom}(f(p/q))\ll q^{N}$ for some absolute $N$, then 
    $f$ has Maillet's property.
    This property fails over the matrix ring. Indeed,
    for any polynomial $f\in\mathbb{Z}[X]$ the according properties
    $\textswab{f}(\mathbb{Q}^{n\times n})\subseteq \mathbb{Q}^{n\times n}$ 
    and $\rm{denom}(\textswab{f}(A/q))\ll q^N$ for $A\in\mathbb{Z}^{n\times n}$
    hold (with $N=\deg f$), but 
    $f$ does not have Maillet's property by Theorem~\ref{1} if $\deg f\ge 2$.
    The reason is basically that
    $A$ being Liouville according to Definition~\ref{lio} does not imply that $A$ is well approximable by rational matrices $B/q$, $B\in\mathbb{Z}^{n\times n}$, $q\in\mathbb{N}$, as a function of $q$ with respect to supremum norm on $\mathbb{R}^{n\times n}$, see also
    Proposition~\ref{hauser}.
\end{remark}

The claim of Corollary~\ref{c} holds for
any transcendental analytic function defined on any open neighborhood of $0$.
Since the setting above is almost completely solved, 
we ask the following variants of Mahler's question for wider classes of functions.

\begin{problem}
   Characterize all functions
    $f: \mathbb{R}^{m\times n}\to \mathbb{R}^{m\times n}$  with the Maillet property that are (a) continuous, (b) analytic 
   in the sense that each of the $mn$ output entry functions is a power series
in the $mn$ entries of $A$.
\end{problem}

For $m=n=1$ part (b) just reduces to Problem~\ref{111}. Any piecewise defined continuous function that is locally a rational function function
with rational coefficients is an example for (a). Possibly the set for part (a) is too large to allow a natural
classification. 

\section{ Continuous images of Liouville
matrices }  \label{cfu}

\subsection{ A converse property for analytic functions and a conjecture }

By contrast to
the results in~\S~\ref{s02}, while not all,
there is still a large subset of quadratic Liouville matrices of any dimension 
that are mapped to Liouville matrices
under any given non-constant analytic function $f$ (more precisely by its extension $\textswab{f}$). Indeed, via a diagonalization argument and a result
from~\cite{asa} for $n=1$ (see~\S~\ref{intro}),
we obtain the following.

\begin{theorem} \label{t2}
	Let $n\ge 2$ an integer and $I\subseteq \mathbb{R}$ an open interval containing $0$.  Let $f_k: I\to \mathbb{R}$ be any sequence of non-constant real analytic functions
    on $I$ and $\textswab{f}_k$ their extensions  
	via \eqref{eq:ex}. Then there is a 
set $\Omega\subseteq \mathscr{L}_{n,n}\subseteq \mathbb{R}^{n^2}$ of Hausdorff dimension $(n-1)^2$ so that  for any $A\in \Omega$ and any $k\ge 1$, we have $\textswab{f}_k(A)\in \mathscr{L}_{n,n}$. In other words the set
\begin{equation} \label{eq:T}
\bigcap_{k\ge 1} \textswab{f}_k^{-1}(\mathscr{L}_{n,n})\cap \mathscr{L}_{n,n} 
\end{equation}
has Hausdorff dimension at least $(n-1)^2$, in particular it is not empty.
\end{theorem}

While the main idea of the proof is 
simple, the condition \eqref{eq:tiere}
causes some technicality, so we move it to~\S~\ref{proof}.
 The lower bound on the Hausdorff dimension is presumably again not optimal, possibly it is the same value $n(n-1)$ as for the full
set $\mathscr{L}_{n,n}$, see for example~\cite{bv} for a considerably stronger result.
The set \eqref{eq:T} is also large in a topological sense, see~\S~\ref{oto} below.

Combination of Theorem~\ref{1} and Theorem~\ref{t2} suggests
the following conjecture, in the spirit of the open problem closing
Burger's paper~\cite{burger}
or~\cite[Theorem~6.2]{ds}.

\begin{conjecture}  \label{KK}
Let $n, I$ be as above.
Suppose $f_k: I\mapsto \mathbb{R}$ and $g_k: I\mapsto \mathbb{R}$, $k\ge 1$, are sequences of analytic functions, with the properties
that
\begin{itemize}
    \item[(i)]  for any $k\ge 1$, the functions $g_k$ are
not of the form \eqref{eq:four} for any $a,b,c,d\in\mathbb{R}$
\item[(ii)] 
for $\tau_{a,b,c,d}$ as in \eqref{eq:four} with any $a,b,c,d\in\mathbb{R}$, we have the non-identity of functions
on $I$
\[
f_{k_1}(z)\ne \tau_{a,b,c,d}(g_{k_2})(z), \qquad k_1, k_2\in \mathbb{N}.
\]
\end{itemize}
Then there exist $A\in \mathscr{L}_{n,n}$ such that $\textswab{f}_k(A)\in \mathscr{L}_{n,n}$ and $\textswab{g}_k(A)\notin \mathscr{L}_{n,n}$ for any $k\ge 1$, with definitions as in \eqref{eq:ex}. Equivalently
\[
\bigcap_{k\ge 1} (\textswab{f}_k^{-1}(\mathscr{L}_{n,n})\cap \textswab{g}_k^{-1}(\mathscr{L}_{n,n}^c))\cap \mathscr{L}_{n,n}\ne \emptyset.
\]
\end{conjecture}

The assumptions are natural in view of our results and cannot be
relaxed up to possibly restricting $a,b,c,d$ to subsets of $\mathbb{R}$.
 We point out that if the conjecture holds
 for some $n=\ell$, then also
 for any $n\ge \ell$. This can be shown by considering 
 $A=\rm{diag}(A_{\ell},B)$ with $A_{\ell}\in \mathscr{L}_{\ell,\ell}$ any such matrix and $B$
 any $(n-\ell)\times (n-\ell)$ matrix so that $\textswab{g}_{\ell}(B)\notin \mathscr{L}_{n-\ell,n-\ell}$ for every $k\ge 1$. Such 
matrices $B$ are easily seen to exist by metrical means, see
the proof of Theorem~\ref{1} below
for more details.
On the other hand,
even for just one pair of functions $f=f_1, g=g_1$ satisfying the hypotheses (i), (ii) with $k=k_1=k_2=1$, the claim of 
Conjecture~\ref{KK} is far from obvious. 

When $n=1$, for a similar claim we would certainly need
to exclude more relations in view of Maillet's result, and a complete description requires a comprehensive understanding of Mahler's problem in the original formulation.

\subsection{ One-to-one continuous maps } \label{oto}

We now study the image of Liouville matrices under functions
$f_k: U\subseteq \mathbb{R}^{m\times n}\to \mathbb{R}^{m\times n}$. Following
the short topological argument in~\cite{sch},
a similar result to Theorem~\ref{t2} can be obtained for injective continuous maps. 

\begin{theorem}  \label{gdel}
    Let $m,n$ be positive integers.
    Let $U\subseteq \mathbb{R}^{m\times n}$ be a non-empty open set
    and $f_k: U\to \mathbb{R}^{m\times n}$ 
a sequence of injective, continuous functions. Then there exists a dense $G_{\delta}$ subset $\Omega\subseteq U\cap \mathscr{L}_{m,n}$ within $U$
such that $f_k(A)\in \mathscr{L}_{m,n}$ for
any $A\in \Omega$ and any $k\ge 1$.
 In other words
\[
\bigcap_{k\ge 1} f_k^{-1}(\mathscr{L}_{m,n})\cap \mathscr{L}_{m,n} 
\]
is a dense $G_{\delta}$ subset of $U$.
\end{theorem}

In the proof we use the famous result of Brouwer that injective
continuous self-maps on a Euclidean space are open onto their image and thus locally induce homeomorphisms.

\begin{proof}
    As indicated we use an analogous argument as in~\cite{sch}.
First notice that, similar to $n=1$, 
for any integer $h\ge 1$
\[
Y_h:= \bigcup_{\textbf{q}\in \mathbb{Z}^n: \Vert \textbf{q}\Vert\ge 2 } \bigcup_{\textbf{p}\in \mathbb{Z}^m} Z_h(\textbf{p},\textbf{q}), \qquad Z_h(\textbf{p},\textbf{q}):=  
\{  A\in \mathbb{R}^{m\times n}: 0<\Vert A\textbf{q}-\textbf{p}\Vert < \Vert \textbf{q}\Vert^{-h}\} 
\]
is an open dense set in $\mathbb{R}^{mn}$. Indeed, it is open by continuity of the maps
$\varphi_{\textbf{p},\textbf{q}}: A\mapsto A\textbf{p}-\textbf{q}$,
and it is dense since for any $A\in \mathbb{Q}^{m\times n}$ there are obviously $\textbf{p}\in\mathbb{Z}^m, \textbf{q}\in\mathbb{Z}^n\setminus\{ \textbf{0}\}$
so that $A\textbf{q}-\textbf{p}=\textbf{0}$ and again by continuity of $\varphi_{\textbf{p},\textbf{q}}\not\equiv \textbf{0}$.
So the intersection
of $Y_h$ over $h\ge 1$ is a 
dense $G_{\delta}$ set. Moreover
if $(L_h)_{h\ge 1}$ denotes the countable collection of rational hyperplanes
in $\mathbb{R}^{mn}$ then $\cap_{h\ge 1} L_h^c$
is obviously a dense $G_{\delta}$ set as well. Hence
\[
\mathscr{L}_{m,n} \supseteq
\bigcap_{h\ge 1} L_h^c \cap Y_h 
\]
is again dense $G_{\delta}$ (there is equality in the inclusion if we restrict to a subset of $L_h$ inducing $\mathbb{Z}$-dependent columns), so $\mathscr{L}_{m,n} \cap U$ is dense $G_{\delta}$ within $U$.

The remaining, short argument based
on Baire's Theorem is precisely as for $m=n=1$ in~\cite{sch}:
Since $f_k$ are injective by Brouwer's result we have that
the images $U_k:= f_k(U)$ are open sets,
hence we find a dense $G_{\delta}$-subset
of $\mathscr{L}_{m,n}$ in each $U_k$,
say $\emptyset\ne \Lambda_k:= U_k\cap \mathscr{L}_{m,n}$.
But since $f_k$ induce homeomorphisms 
this means their preimages $Z_k=f_k^{-1}(\Lambda_k)$ are again dense $G_{\delta}$-sets
(in $U$).
Hence the countable intersection
$\Omega:=\cap_{k\ge 1} Z_k\cap \mathscr{L}_{m,n}$ is a dense 
$G_{\delta}$ subset of $U$ as well.
Any matrix in this set $\Omega$ has the claimed property.
\end{proof}

The same topological
result can be obtained 
in the setting of
Theorem~\ref{t2} as well,
meaning \eqref{eq:T} is again a dense $G_{\delta}$ set for $\textswab{f}_k$ derived from non-constant scalar analytic $f_k$ via \eqref{eq:ex}.
However, as it cannot be directly deduced from Theorem~\ref{gdel}
and a complete proof we found is slightly technical, we prefer to omit it here. (For instance, a technical problem is that derivatives $f_k^{\prime}$ may vanish within $I$, then $f_k$ and $\textswab{f}_k$ are not locally injective.)
On the other hand, it is not clear to us if a positive Hausdorff dimension result can be obtained in context of
Theorem~\ref{gdel}.



A special case is the following generalization of a result by 
Erd\H{o}s~\cite{erd} for $m=n=1$.

\begin{corollary} \label{KKo}
    Any $A\in \mathbb{R}^{m\times n}$ can be written $A=B+C$ with $B,C\in \mathscr{L}_{m,n}$.
\end{corollary}

\begin{proof}
    Apply Theorem~\ref{gdel}
    with $f(X)=A-X$.
\end{proof}

We should notice that the ideas in all
papers~\cite{asa, sch, wald} as well as our proof of Theorem~\ref{gdel} above originate in this work of
Erd\H{o}s. See also~\cite{rie}.
Via Theorem~\ref{gdel},
we can further directly extend several results from~\cite{wald} to the matrix setting, we only state the analogue of~\cite[Corollary~7]{wald}.

\begin{corollary}
    Let $U\subseteq \mathbb{R}^{m\times n}$
    be a non-empty open connected set,
    and $\varphi: U\to U$ be an injective, continuous self map. Then there exists a dense $G_{\delta}$ set of $A\in \mathscr{L}_{m,n}\cap U$ so that the
    orbit $\varphi^k(A)=\varphi\circ \varphi\cdots\circ\varphi(A)$, $k\ge 1$, consists only
    of elements in $\mathscr{L}_{m,n}\cap U$.
\end{corollary}

In the original formulation in~\cite{wald} it is assumed that 
$\varphi$ is bijective, however surjectivity is not needed.

\begin{proof}
    Apply Theorem~\ref{gdel} to
    $f_k=\varphi^k$, which are defined and easily seen to inherit the properties of being continuous and injective
    from $\varphi$.
\end{proof}

\subsection{On relaxing conditions
of Theorem~\ref{gdel} }
For $m=n=1$, the assumption of injectivity in Theorem~\ref{gdel} can be relaxed. As stated before, indeed it was shown in~\cite{asa} that we only need
$f_k$ to be nowhere constant on an interval $I$,
meaning not constant on any non-empty open subinterval of $I$,
for the implication of Theorem~\ref{gdel}. The latter indeed 
defines a strictly larger set of functions. 
In higher dimensions, a priori 
the most 
natural way to extend the concept of nowhere constant seems the following.

\begin{definition} \label{defi2}
    Let $m_1,m_2,n_1,n_2$ be positive integers and
    $U\subseteq \mathbb{R}^{m_1\times n_1}\cong \mathbb{R}^{m_1n_1}$ be an open, non-empty set. We call a matrix function
    $f: U\to \mathbb{R}^{m_2\times n_2}$ {\em nowhere constant} if it is not constant
   on any non-empty open subset of $U$.
\end{definition}

 When $n\ge 2$, 
it is easy to see that the analogue of~\cite{asa} fails. 
Consider the
functions that maps $A=(a_{i,j})_{1\le i\le m,1\le j\le n}\in U$ to the $m\times n$ matrix $f(A)=B=(b_{i,j})_{1\le i\le m,1\le j\le n}$ with each entry $b_{i,j}=a_{1,1}$. Indeed, this function is continuous, nowhere constant but any point in the image satisfies a fixed linear dependence of columns relation over $\mathbb{Z}$, hence 
$f(U)$ has empty intersection in a trivial way 
with $\mathscr{L}_{m,n}$ as \eqref{eq:tiere} fails for any $B\in f(U)$. So it seems reasonable to also consider the following additional property. 

\begin{definition}
    A function $f: U\subseteq \mathbb{R}^{m\times n}\to \mathbb{R}^{m\times n}$ satisfies the $LIC$ property (linearly independent columns) if for every set $\Omega\subseteq U$ with non-empty interior,
    there
    are no fixed $\textbf{p}\in\mathbb{Z}^m, \textbf{q}\in\mathbb{Z}^n\setminus \{ \textbf{0}\}$ so that $B\textbf{q}-\textbf{p}=\textbf{0}$ for any $B\in f(\Omega)$.
\end{definition}

The LIC property is easily seen to be independent from the condition of being nowhere constant. 
However, even assuming both nowhere constant and LIC property is not enough to guarantee the conclusion of Theorem~\ref{gdel} in general.

\begin{theorem}  \label{ier}
Let $m,n$ be positive integers and $U\subseteq \mathbb{R}^{m\times n}$ be an open, non-empty set.
\begin{itemize}
     \item[(i)] If $m=n=1$, then for every
     nowhere constant continuous function $f: U\to \mathbb{R}$ the set
     $f^{-1}(\mathscr{L}_{1,1})\cap \mathscr{L}_{1,1}$ is dense $G_{\delta}$ in $U$, in particular non-empty.
     \item[(ii)]
    If $(m,n)\ne (1,1)$, then there exists
    a nowhere constant continuous function 
    $f: U\to \mathbb{R}^{m\times n}$ so that $f(A)\notin \mathscr{L}_{m,n}$ for any matrix $A\in U$, so $f^{-1}(\mathscr{L}_{m,n})=\emptyset$.
     
    \item[(iii)] If $m\ge 2$, there
    exists a nowhere constant continuous function with LIC property 
    $f: V\to \mathbb{R}^{m\times n}$ so that $f(A)\notin \mathscr{L}_{m,n}$ for any matrix $A\in U$, so $f^{-1}(\mathscr{L}_{m,n})=\emptyset$.
    \end{itemize}
\end{theorem}

\begin{proof}
Claim (i) is just~\cite{asa}, if $n\ge 2$ claim (ii) has already been observed above.  
Note that (iii) also contains the remaining case $n=1, m\ge 2$ of (ii). So 
to prove (iii), let
$m\ge 2$ and consider a function $f$ that acts as the identity on the first line $\textbf{a}_1=(a_{1,1},\ldots,a_{1,n})$ of a matrix $A=(a_{i,j})\in U$ 
    and constant on the remaining lines. Hereby we choose the constant image vectors $\textbf{b}_2,\ldots,\textbf{b}_m\in \mathbb{R}^n\setminus (\mathscr{L}_{1,n}\cup \Pi(L_j))$ where $L_j$ are the countable collection of all rational hyperplanes of 
    $\mathbb{R}^{n+1}$ and $\Pi: \mathbb{R}^{n+1}\to \mathbb{R}^n$ is the restriction by chopping off the last coordinate.
    Such $\textbf{b}_i$ clearly exist, in fact this set has full $n$-dimensional Lebesgue measure as
    $\mathscr{L}_{m,n}$ and $\Pi(L_j)$ are all nullsets.
  Then $f$ is clearly continuous and nowhere constant, and satisfies the LIC property since $\textbf{b}_2\notin \cup \Pi(L_j)$ implies
  \eqref{eq:tiere} for any matrix $B\in f(U)$.
    However, since the second line $\textbf{b}_2\notin \mathscr{L}_{1,n}$, we conclude
    $f(A)\notin \mathscr{L}_{m,n}$  by Proposition~\ref{hauser}, indeed $\infty>\omega^{2\times n}(\textbf{b}_2)\ge \omega^{m\times n}( f(A))$.  
\end{proof}

Observe that claim (ii), (iii)
apply to all matrices $A$, not only
Liouville matrices. The main reason for the failure in (ii), (iii) for $m\ge 2$ is that in higher dimension, the property nowhere constant for a function is insufficient to guarantee that its image contains an open set, which was used in the argument for $m=n=1$ in~\cite{asa}.

There is a gap in Theorem~\ref{ier}
for $m=1, n\ge 2$ and nowhere constant, continuous functions with LIC property. The converse result of the next theorem illustrates
why this case is more complicated.
Define $\mathscr{L}^{\ast}_{m,n}\supseteq \mathscr{L}_{m,n}$ 
the set of irrational real matrices $A\notin \mathbb{Q}^{m\times n}$ such that $\Vert A\textbf{q}-\textbf{p}\Vert< \Vert \textbf{q}\Vert^{-N}$ has a solution in integer
vectors $\textbf{p}\in\mathbb{Z}^m,\textbf{q}\in\mathbb{Z}^n\setminus \{ \textbf{0}\}$ for any $N$. As indicated in~\S~\ref{intro}, for $n=1$ we have $\mathscr{L}^{\ast}_{m,1}=\mathscr{L}_{m,1}$, but
for $n\ge 2$ the inclusion is proper.

\begin{theorem}  \label{neu}
     Let $m=1$ and $U\subseteq \mathbb{R}^{1\times n}$ be an open, non-empty set.
      For any sequence
    of nowhere constant continuous functions $f_k: U\to \mathbb{R}^{1\times n}$ there exists a dense $G_{\delta}$ set $\Lambda\subseteq \mathscr{L}_{1,n}\cap U$ 
    so that $f_k(A)\in\mathscr{L}_{1,n}^{\ast}$ 
    for any $A\in \Lambda$ and every $k\ge 1$. In other words
    \[
    \bigcap_{k\ge 1} f_k^{-1}(\mathscr{L}_{1,n}^{\ast})\cap \mathscr{L}_{1,n}
    \]
    is dense a $G_{\delta}$ 
    subset of $U$.
\end{theorem}

We believe the claim remains true for $\mathscr{L}_{1,n}$ in place of $\mathscr{L}_{1,n}^{\ast}$ throughout,
a proof would be desirable.
On the other hand, for $m\ge 2$ the analogue of Theorem~\ref{ier}, (iii), holds as well for $\mathscr{L}^{\ast}_{m,n}$ by the same proof, which is formally a stronger claim.

\begin{proof}
      We can assume $n\ge 2$, the case $m=n=1$
    reduces to~\cite{asa}. 
    First note that by Propositon~\ref{hauser}, case $m=1$,
    a line vector belongs to 
    $\mathscr{L}_{1,n}^{\ast}$
    (but not necessarily $\mathscr{L}_{1,n}$)
    as
    soon as some entry is a Liouville number. 
    Hence 
    for any non-empty open interval $I$, if $L_I=\mathscr{L}_{1,1}\cap I\subseteq I$ denotes
    the dense $G_{\delta}$ set of Liouville numbers in $I$,
    the cylinder sets
    \[
    Z_j=\mathbb{R}^{j-1}\times L_I\times \mathbb{R}^{n-j}, \qquad 1\le j\le n,
    \]
    consist only of elements of $\mathscr{L}_{1,n}^{\ast}$, so
    \[
    Z_j\subseteq \mathscr{L}_{1,n}^{\ast}.
    \]
    Take any non-empty open box $\mathcal{B}$ 
    in $U\subseteq \mathbb{R}^n$. First consider just one function $f=f_1$.
    Then since $f=(f^1,\ldots,f^n)$ is nowhere constant
    some coordinate function $f^j$ is not constant on $\mathcal{B}$. So
    since connectedness is preserved under continuous maps,
    its image $f^j(\mathcal{B})\subseteq \mathbb{R}$ contains a non-empty open interval $I$.
    But this means $f(\mathcal{B})$ has non-empty intersection with the cylinder set $Z_j$ above, or 
    equivalently $f^{-1}(Z_j)\cap \mathcal{B}\subseteq f^{-1}(\mathscr{L}_{1,n}^{\ast})\cap \mathcal{B}$ is non-empty.
    Thus, as $\mathcal{B}$ was arbitrary in $U$, the set $f^{-1}(\mathscr{L}_{1,n}^{\ast})$ is dense in $U$. Moreover, as $f$ is continuous
    and $\mathscr{L}_{1,n}^{\ast}\supseteq \mathscr{L}_{1,n}$
    is dense $G_{\delta}$ (see the proof of Theorem~\ref{gdel}), $f^{-1}(\mathscr{L}_{1,n}^{\ast})$ is a $G_{\delta}$ set. Hence
    it is a dense $G_{\delta}$ set.
     Now we apply this to $f=f_k$ for all $k\ge 1$ simultaneously, 
     and again since $\mathscr{L}_{1,n}$
     is dense $G_{\delta}$ as well
     we see that the set
     \[
     \Lambda:= \bigcap_{k\ge 1} f_k^{-1}(\mathscr{L}_{1,n}^{\ast}) \cap \mathscr{L}_{1,n}
     \]
    is a dense $G_{\delta}$ set in $U\subseteq \mathbb{R}^n$ as well, with the property of the theorem that $f_k(A)\in \mathscr{L}_{1,n}^{\ast}$ for any $A\in \Lambda\subseteq \mathscr{L}_{1,n}$ and every $k\ge 1$.
\end{proof}

Even if assuming LIC property, the counterexamples in Theorem~\ref{ier} (ii), (iii) are still
slightly artificial as the image has certain constant entry functions.
To avoid this, we propose an alternative to our definition of a nowhere constant function.
Since, very similarly to Proposition~\ref{pro}, the Liouville property of $A\in \mathbb{R}^{m\times n}$ is preserved under actions 
\begin{equation}  \label{eq:Ar}
    A\mapsto R_1AR_2+T,\qquad R_1\in \mathbb{Q}^{m\times m}, R_2\in \mathbb{Q}^{n\times n}, 
    T\in \mathbb{Q}^{m\times n}
\end{equation}
with invertible matrices $R_j$, the following strengthening seems natural.

\begin{definition}
We refer to $A,B\in \mathbb{R}^{m\times n}$ as $\mathscr{L}$-equivalent and write $A\sim B$ if $B$ arises from $A$
via \eqref{eq:Ar}. We write $[A]_{\sim}$ for the class of $A$. We say functions $f,g: U\subseteq \mathbb{R}^{m\times n}\to \mathbb{R}^{m\times n}$ are $\mathscr{L}$-equivalent
and write $f\sim g$ if $g(A)=R_1 f(A)R_2+T$ for fixed $R_j, T$ as in \eqref{eq:Ar} and any $A\in U$, and again write $[f]_{\sim}$ for the class of $f$. We finally call
$f$ {\em strongly 
    nowhere constant} if for 
    every $g\in [f]_{\sim}$ every scalar entry function $g^{i,j}: U\to \mathbb{R}$, $1\le i\le m, 1\le j\le n$, is nowhere constant in the sense of Definition~\ref{defi2}.
\end{definition}

Plainly, we can restrict to  $g\in [f]_{\sim}$ derived via $T=\textbf{0}$ for testing the property
strongly nowhere constant.
Thus the property means that the $\mathbb{Z}$-span of the entries $f^{i,j}(A)$ of $f(A)$ together with the $\textbf{1}$ function can only generate the $\textbf{0}$ function in the trivial way.
For strongly nowhere constant functions, we prove the following
result on column vectors.

\begin{theorem}  \label{analog}
    Let $m\ge 2$, $n=1$. Then there
    exists a strongly nowhere constant
    continuous function $f: \mathbb{R}^{m\times 1}\to \mathbb{R}^{m\times 1}$ with LIC property and so that $f(A)\notin \mathscr{L}_{m,1}$ for any $A\in \mathbb{R}^{m\times 1}$, in other words $f^{-1}(\mathscr{L}_{m,1})=\emptyset$.
\end{theorem}

This is stronger than Theorem~\ref{ier}, (ii), (iii), 
for the special case $m\ge 2, n=1$.
The principal idea of the proof is to consider a function
whose image is contained in an algebraic variety without rational points, which by
a result in~\cite{iche} means they contain no Liouville (column) vectors.

\begin{proof}
    Write $A=(a_1,\ldots,a_m)^{t}$ and define the coordinate functions of $f=(f^1,\ldots,f^m)^{t}$ by 
    \[
    f^j(a_{1}, \ldots, a_{m})=a_{j},\; (1\le j\le m-1),\qquad f^m(a_{1},\ldots, a_{m})=\sqrt[3]{a_{m-1}^2-N}
    \]
    where $N\in\mathbb{Z}$ is so that $Y^2=X^3+N$
    has no rational solution (see~\cite{horie} for existence).
   It is not hard to see that this induces a continuous, strongly nowhere constant function on $\mathbb{R}^m$ with LIC property.
Consider the projection 
$\Pi:\mathbb{R}^{m\times 1}\to \mathbb{R}^{2\times 1}$ to the last two coordinates, so that $\Pi(f(\mathbb{R}^{m\times 1}))$ equals the variety $X_{m}^3+N=X_{m-1}^2$ defined over $\mathbb{Q}[X_{m-1},X_m]$, without rational points. As there is no rational point 
    in $\Pi(f(\mathbb{R}^{m\times 1}))\subseteq \mathbb{R}^{2\times 1}$, from \cite[Theorem~2.1]{iche}
    we see that any such projected
    vector $\textbf{a}\in \Pi(f(\mathbb{R}^{m\times 1}))\subseteq \mathbb{R}^2$ has irrationality exponent $\omega^{2\times 1}(\textbf{a})\le 2$, hence $\textbf{a}\notin \mathscr{L}_{2,1}$. However, as clearly the irrationality exponent of a column vector cannot decrease under projection 
    $\Pi$ (see Proposition~\ref{hauser}),
    a fortiori any $\textbf{b}\in \Pi^{-1}(\textbf{a})$ has exponent at most $\omega^{m\times 1}(\textbf{b})\le \omega^{2\times 1}(\textbf{a})\le 2$ as well, thus $\textbf{b}\notin \mathscr{L}_{m,1}$. So since $\Pi^{-1}(\Pi(f(\mathbb{R}^{m\times 1}))\supseteq f(\mathbb{R}^{m\times 1})$ and $\textbf{a}\in \Pi(f(\mathbb{R}^{m\times 1}))$ was arbitrary, we have $f(\mathbb{R}^{m\times 1})\cap \mathscr{L}_{m,1}=\emptyset$, equivalent to the claim.
    \end{proof}

    \begin{remark}
    The function $f^m(a_{1}, \ldots, a_{m})=\sqrt{3-a_{m-1}^2}$ is an example of lower degree due to $X^2+Y^2=3$
    having no rational points again, however it is not globally defined.
    \end{remark}

    For $m=1$, a converse result in form of the analogue of Theorem~\ref{neu}
    clearly holds for strongly nowhere constant functions a fortiori, and presumably also 
    for $\mathscr{L}_{1,n}$ (without ``star''). This leaves the non-vector cases $\min\{m,n\}>1$ open where the situation seems unclear.

    \begin{problem}
        Determine all pairs $m,n$ for which the analogue of Theorem~\ref{analog} holds.
    \end{problem}

Above we have only considered self-mappings.
We want to finish this section with an open problem for a function
$f:\mathbb{R}\to \mathbb{R}^{n\times n}$.

\begin{problem} \label{p2}
    Given $A\in\mathbb{R}^{n\times n}$, what can be said about the set
    \[
    \mathcal{X}_A=\{ \lambda\in \mathbb{R}: A-\lambda I_n \in \mathscr{L}_{n,n}\}.
    \]
    Is it always non-empty/uncountable/dense $G_{\delta}$?
\end{problem}

It is not hard to see by topological arguments similar to Theorem~\ref{gdel} that for any $A$, all lines of $A-\lambda I_n$
can be made Liouville vectors in $\mathbb{R}^{1\times n}$ simultaneously for many (dense $G_{\delta}$ set of) 
$\lambda$, however this is in general insufficient for $A-\lambda I_n$ being in $\mathscr{L}_{n,n}$. This argument works if $A$ is diagonalizable via a rational base change matrix (in particular if $A$ is diagonal), via Proposition~\ref{pro}.
On the other hand, it is easy to construct $A$ such that for no $\lambda$ any column of $A-\lambda I_n$ is a Liouville vector in $\mathbb{R}^{m\times 1}$, however this is not necessary for $A-\lambda I_n$ being in $\mathscr{L}_{n,n}$. If $\mathcal{X}_A$ is non-empty, it is dense in $\mathbb{R}$ by its invariance under rational translation via Proposition~\ref{pro} again.
In general,
we do not know what to conjecture, but remark that Problem~\ref{p2} can be naturally generalized to $\mathcal{X}_{A,B}=\{ \lambda\in \mathbb{R}: A-\lambda B \in \mathscr{L}_{m,n}\}$ for $A,B\in \mathbb{R}^{m\times n}$.

\section{On a property of Burger}
Let us return to quadratic matrices. 
As noticed in~\S~\ref{oto}, it was shown by Erd\H{o}s~\cite{erd}
and is a special case of Corollary~\ref{KKo} that any real number can be written as the sum of two Liouville numbers.
Burger~\cite{burger} noticed that Erd\H{o}s' proof
can be adapted to show that any
transcendental real number can
be written as a sum of two algebraically independent Liouville numbers. The converse is easy to prove, giving a characterization of transcendental real numbers. 
See also~\cite[Propositon~3]{wald} for a generalization.
Let us consider the problem in the matrix setting. Naturally we call a quadratic real matrix algebraic (over $\mathbb{Z}$) if $P(A)=\textbf{0}$ for some non-zero polynomial $P\in \mathbb{Z}[X]$, otherwise transcendental. Similarly let us call two real $n\times n$ matrices
$B,C$ algebraically dependent (over $\mathbb{Z}$) if
there exists non-zero $P\in \mathbb{Z}[X,Y]$ so that
$P(B,C)=\textbf{0}$, otherwise algebraically independent. Note hereby that bivariate polynomials over the 
non-commutative matrix ring 
have a more complicated form as in the commutative case $n=1$, for
example $5X^{2}Y^{3}X-5X^3Y^3$ is not the $0$ polynomial
when $n\ge 2$. 
We study Burger's property in this setting 

\begin{problem} \label{p5}
  For $n\ge 2$, is it true that $A\in \mathbb{R}^{n\times n}$ is transcendental
if and only if it has a representation as a sum of two
algebraically independent Liouville matrices?
\end{problem}

For $n=2$, Problem~\ref{p5} has a negative answer
in a trivial sense,
as it turns out that

\begin{proposition}  \label{pp2}
    For $n\ge 2$, there are no algebraically independent matrix pairs at all.
\end{proposition}

\begin{proof}
 For $n=2$ this is due to Hall's identity
\[
X(YZ-ZY)^2= (YZ-ZY)^2X
\]
for any integer (or real) $2\times 2$ matrices $X,Y,Z$. Indeed it suffices for example to let $X=Y=A, Z=B$ to see
that any two $2\times 2$ matrices $A, B$ are algebraically dependent according to our definition above.
It appears that similar examples can be found for general $n\ge 2$, 
as the matrix rings are all so-called
polynomial identity rings,
by the Amitsur–Levitzki Theorem. 
We refer to~\cite{dre}.
\end{proof}

It appears that the ordinary concept of algebraic independence is too strong over our matrix rings. However 
if we modify our definitions of transcendence and algebraic independence, Burger's problem becomes
more interesting. We propose to work with the following concepts.

\begin{definition}
    \begin{itemize}
        \item[(i)] 
    Call $A\in \mathbb{R}^{n\times n}$  {\em weakly algebraic} if there
    exists an integer $\ell\ge 0$ and a polynomial $P\in \mathbb{Z}[X_0,\ldots,X_{\ell}]$ so that 
    \[
    P(A,B_1,\ldots,B_{\ell})=\textbf{0}
    \]
    for any $B_1,\ldots,B_{\ell}\in \mathbb{R}^{n\times n}$ but
    \[
    P(C_0,\ldots,C_{\ell})\ne \textbf{0}
    \]
    for some $C_0,\ldots,C_{\ell}\in \mathbb{R}^{n\times n}$. 
\item[(ii)] Call $A,B\in \mathbb{R}^{n\times n}$ {\em weakly algebraically independent} if $P(A,B)=\textbf{0}$
for $P\in \mathbb{Z}[X,Y]$
implies $P(C_0,C_1)=\textbf{0}$ for all matrices $C_0,C_1\in \mathbb{R}^{n\times n}$.
    \end{itemize}
\end{definition}

We only want the polynomials as functions to act non-trivially in (i),(ii), essentially we factor out the
non-trivial polynomial relations in $\mathbb{Z}[X,Y]$ over the matrix ring.
Algebraic
implies weakly algebraic
as we may let $\ell=0$ and
the ring $\mathbb{Z}[X]$ in a single real $n\times n$, $n\ge 2$, 
matrix variable is not 
a polynomial identity ring (it is a principal ideal domain). The implication follows also from Theorem~\ref{subring} below. The concepts algebraic and weakly algebraic may however be equivalent,
part of Problem~\ref{ppr} below.
Moreover, it is immediate that weakly algebraically independent
implies algebraically independent in the classical sense. For $n=1$ the respective concepts coincide, if
we assume working over a commutative 
ring (i.e. identifying $AB$ with $BA$). On the other hand

\begin{proposition}
   For $n\ge 2$,
weakly algebraically independent matrix pairs exist.
\end{proposition}
 Thus, comparing with Proposition~\ref{pp2}, weakly algebraically independent is a strictly weaker concept than algebraically independent. 
 
 \begin{proof}
     Fix for now 
$P\in \mathbb{Z}[X,Y]$ that does not induce the $\textbf{0}$ function. Then,
as any entry $P_{i,j}(A,B)$, $1\le i,j\le n$, of $P(A,B)$ is a scalar-valued multivariate polynomial in the entries of $A=(a_{i,j})$ and $B=(b_{i,j})$
and some $P_{i_0,j_0}(A,B)$ is not identical $0$, we find many (full Lebesgue measure set in $\mathbb{R}^{2n^2}$) real matrix pairs $A,B$
so that $P_{i_0,j_0}(A,B)\ne 0$ thus also $P(A,B)\ne \textbf{0}$. Since
$\mathbb{Z}[X,Y]$ is only countable
we are still left with a full measure set
that avoids all algebraic varieties.
 \end{proof}

Alternatively the claim follows from the next Theorem~\ref{subring}.
It generalizes Burger's result,
with similar underlying proof ideas. However,
we use a topological argument rather
than a cardinality consideration as in~\cite{burger}, and a few more twists.

\begin{theorem} \label{subring}
    Let $n\ge 1$ and $A\in\mathbb{R}^{n\times n}$. Then
    \begin{itemize}
      \item[(i)] if $A=B+C$
 holds for some weakly algebraically independent $B,C\in\mathbb{R}^{n\times n}$, then $A$ is transcendental.
        \item[(ii)] if $A$ is not weakly algebraic (in particular then $A$ is transcendental) then
there exist weakly algebraically independent $B,C\in\mathscr{L}_{n,n}$ with $B+C=A$. 
    \end{itemize}

\end{theorem}

In fact in (ii) we only need to assume $A$ not weakly algebraic for 
$\ell=1$, a stronger claim. As stated above,
for $n=1$ we get a new proof of 
Burger's result avoiding Bezout's Theorem. However, for $n\ge 2$, (i), (ii) do not give rise to any logical equivalence.
Consider the claims: (I) $A$ is
    transcendental.
    (II) $A$ is not weakly algebraic. (III) $A$ can be written $A=B+C$ with $B,C$ weakly algebraically independent. Then Theorem~\ref{subring} is equivalent to $(II)\Longrightarrow (III)\Longrightarrow (I)$. 
Indeed, the following remains open.

\begin{problem} \label{ppr}
    When $n\ge 2$, is it actually true that $(I)\Longleftrightarrow (III)$, or even
    $(I)\Longleftrightarrow (II)\Longleftrightarrow (III)$? 
\end{problem}

Clearly if (and only if) weakly algebraic is actually the same as algebraic, then we have the full equivalence.
Possibly generalizations of Theorem~\ref{subring} as in~\cite[Proposition~3]{wald} to more general expressions (polynomials) in $B,C$ in place of the plain sum $B+C$ hold, however our proofs of neither (i) nor (ii) extend in an obvious way. Moreover, there may be alternative natural variants of Burger's problem for matrices to the one discussed above worth studying. We stop our investigation here.

\section{Proof of Theorem~\ref{1} }


First let $n=2$.
The principal idea is to consider matrices
of the form
\[
	A=\begin{pmatrix}
	a & b \\
	0 & 0
	\end{pmatrix}
	\] 
	where $b\in \mathscr{L}_{1,1}$ is any fixed Liouville number and $a$ to be chosen later. 
	Then $A\in \mathscr{L}_{2,2}$ for any $a$ not linearly dependent with $\{1,b\}$ over $\mathbb{Z}$,
 since $\vert q b-p\vert=\Vert A\cdot (0,q)^{t}-(0,p)^{t}\Vert$ for any $p,q$.  First consider just one function $f(z)=f_1(z)= \sum_{j\ge 0} c_j z^j$. We will find many $a$
	so that $\textswab{f}(A)= \sum c_j A^j$ is not Liouville. Note first that
		\[
	A^j= \begin{pmatrix}
	a^j & b a^{j-1} \\
	0 & 0
	\end{pmatrix}, \qquad j\ge 1
	\] 
 so
 \[
 \textswab{f}(A)=\begin{pmatrix}
	r(a) & s(a) \\
	0 & c_0
	\end{pmatrix}
 \]
	where 
    \[
    r(z)=\sum_{j\ge 0} c_{j} z^j= f(z), \qquad s(z)= \sum_{j\ge 0} bc_{j+1} z^{j}. 
    \]
    Hence
       \begin{equation} \label{eq:ratz}
    r(z)-c_0 = \frac{1}{b}\cdot s(z)z.
    \end{equation}
	Then for $\textbf{v}_a:=(r(a), s(a))\in \mathbb{R}^{1\times 2}$, Proposition~\ref{hauser} implies 
 $\omega^{2\times 2}(\textswab{f}(A))\le \omega^{1\times 2}(\textbf{v}_a)$, with equality if $c_0=0$ (as the additional Diophantine condition 
 on $|q_2 c_0-p_2|$ holds for $p_2=0$ if $c_0=0$). So it suffices to show $\omega^{1\times 2}(\textbf{v}_a)= 2$ for many $a$. 

    Now  $\mathcal{C}: =\{ (r(a), s(a)): a\in\mathbb{R}\}$ defines an analytic curve in
    $\mathbb{R}^2$. We use a deep result
    from~\cite{km} to conclude.
    Let us call a planar curve locally parametrized by $(x,f(x))$ non-degenerate if the critical points of
    $f^{\prime\prime}(x)$, i.e.
    where it vanishes or does not exist,
    happens only for a set $x$ of Lebesgue measure $0$.

    \begin{theorem}[Kleinbock, Margulis]
        Let $\mathcal{C}$ be a non-degenerate planar curve given by parametrization $(x(t),y(t))$, $t\in I$. Then for almost
        all $t\in I$ with respect to Lebesgue measure we have
        $\omega^{1\times 2}(x(t),y(t))=2$.
    \end{theorem}

    So if $\mathcal{C}$ above is non-degenerate,
    we know that for almost all $a$ the 
    point $\textbf{v}_a$ indeed has Diophantine exponent $2$. By omitting additionally the countable set of $a$ where \eqref{eq:tiere} fails, i.e. excluding elements $\mathbb{Q}$-linearly dependent with $\{ b,1\}$, to make $A$ a Liouville matrix,
    we are still left with a full measure set, so we are done.
    (In fact such $a$ do not exist
    as $\mathbb{Q}$-linear dependence
    directly implies $\omega^{1\times 2}(\textbf{v}_a)=\infty>2$.)
    
    So suppose conversely that $\mathcal{C}$ is degenerate.
    Since $\mathcal{C}$ is defined via analytic entry functions, the zeros
    of the according second derivative
    form a countable discrete set unless constant $0$, so
    the second derivative must vanish everywhere. But this requires
    \[
    \frac{d^2s(z)}{d^2r(z)}= \frac{ d (s_z/r_z) }{dr_z} =
    \frac{\frac{d(s_z/r_z)}{dz} }{\frac{dr}{dz}}=
    \frac{ s_{zz} r_z - s_z r_{zz} }{ r_z^3 }= 0
    \]
    identically or identically undefined.
    If $s_z\equiv 0$ by \eqref{eq:ratz} this yields $f(z)=r(z)=c_0+c_1 z$
    for some real numbers $c_0, c_1$, among
    the functions \eqref{eq:fro} excluded in the theorem. So
    assume $s_z\not\equiv 0$.
    Then $r_z/s_z\equiv d$
    is constant. We may assume $d\ne 0$
    as $d=0$ leads to
    $f(z)=r(z)=c_0$ constant, again excluded.
    Then, again by \eqref{eq:ratz}, the condition becomes
    \[
    \frac{1}{b}\cdot \frac{s_z(z)z+s(z)}{s_z}= \frac{1}{b} (z+s/s_z) \equiv d
    \]
    or $z+s/s_z\equiv g$ for a new constant $g=bd$. This is further equivalent to 
    \[
    (\log s)_z= \frac{s_z}{s} \equiv -\frac{1}{z-g}
    \]
    or $\log s= - \log |z-g|+ h$ for some $h$, finally $s(z)= H/(z-g)$ for $H=e^h$. Hence
    \[
    f(z)=r(z)= \frac{1}{b}\cdot s(z)z+c_0=
    c_0 + J\frac{z}{z-g}= c_0+J+\frac{Jg}{z-g}= \tilde{c}_0+ \frac{ \tilde{c}_1 }{z-g}
    \]
    for some real numbers $J=H/b$ and $\tilde{c}_i=Jg$, again of the form 
    \eqref{eq:fro} excluded in the theorem. The argument for a single function is complete.
%
%

    We can apply this argument for any $f=f_k$. Since we get a full measure set
    of $a\in\mathbb{R}$ for any $k\ge 1$, their countable intersection again has full measure. So $n=2$ is done.
    
For larger $n$, let us for simplicity denote simultaneously
$\textswab{f}_k$ derived from $f_k$ via 
\eqref{eq:ex} for matrices in arbitrary dimension (we will use dimension $2$, $n-2$ and $n$).
Then we can just take a matrix consisting of two diagonal blocks
\[
A= \rm{diag}( A_2 , B )
\]
with $A_2\in \mathscr{L}_{2,2}$ as above and $B$ any real $(n-2)\times (n-2)$ matrix so that 
all $\textswab{f}_k(B)$, $k\ge 1$, have irrationality exponent $\omega^{(n-2)\times (n-2)}(\textswab{f}_k(B))=1$.
For any $k$ it is easily seen
that such $B$ form a full Lebesgue measure
set in $\mathbb{R}^{(n-2)^2}$. Indeed,
this follows
from the locally bi-Lipschitz property of the analytic maps
$B\to \textswab{f}_k(B)$ and a standard Khintchine type result that the set of matrices 
\[
\{ C\in \mathbb{R}^{(n-2)\times (n-2)}:\omega^{(n-2)\times (n-2)}(C)=1\}
\]
has full Lebesgue measure
in $\mathbb{R}^{(n-2)^2}$.
So the same holds for the infinite intersection over $k\ge 1$ as requested. It is easily checked that any resulting $A\in \mathscr{L}_{n,n}$ since $A_2\in \mathscr{L}_{2,2}$
and the system is decoupled. On the other hand for $n\ge 2, k\ge 1$ we have
\[
\omega^{n\times n}(\textswab{f}_k(A))= \max\{ \omega^{2\times 2}(\textswab{f}_k(A_2)), \omega^{(n-2)\times (n-2)}(\textswab{f}_k(B)) \}\le \max\{ 2,1\}= 2,
\]
with equality in the inequality if $f_{k}(0)=0$, where for the first identity again we used that the system is decoupled (see also~\cite[Lemma~9.1]{zite}). So $\textswab{f}_k(A)\notin \mathscr{L}_{n,n}$. 

Finally as the set of suitable $a\in\mathbb{R}$ and $B\in \mathbb{R}^{(n-2)\times (n-2)}$ have full Lebesgue measure in the respective Euclidean spaces, the metrical claim follows from a 
standard estimate on the Hausdorff dimension Cartesian products $\dim_H(X\times Y)\ge \dim_H (X)+\dim_H (Y)$ for measurable $X,Y$, 
see Tricot~\cite{tricot}.

\begin{remark}
    By choosing $b$ an ultra-Liouville 
    number as defined in~\cite{s}, the
    arising matrix $A$ will have an analogous property. This justifies Remark~\ref{remrand}.
\end{remark}

 \begin{remark}
     Alternatively for $n\ge 3$ we can consider $A=\rm{diag}(A_2, A_2,\ldots,A_2)$ for $n$ even
     and $A=\rm{diag}(A_2, A_2,\ldots,A_2, \{ \ell\})$ for $n$ odd,
     where $\ell\in\mathbb{R}$ is a number all of whose evaluations $f_k(\ell)$ have exponent $\omega^{1\times 1}(f_k(\ell))=1$ (which again exists by the same metrical argument as in the proof above). However, this gives a weaker metrical bound $\lceil n/2\rceil$.
 \end{remark}

\section{Proof of Theorem~\ref{t2} } \label{proof}

In view of condition \eqref{eq:tiere},
we need the following technical lemma.
Possibly stronger claims are known, but we have found no reference, so we prove it directly using the metrical sparsity 
of zeros of multivariate real power series and a diagonalization argument.

\begin{lemma} \label{depat}
Let $n\ge 1$ be an integer, $I\ni 0$ be an open interval and $f: I\to \mathbb{R}$ be any non-constant analytic map. Let $\mathcal{S}\subseteq \mathbb{R}^{n\times n}\cong \mathbb{R}^{n^2}$ be a proper affine subspace.
Then for $\textswab{f}: U\subseteq \mathbb{R}^{n\times n}\to \mathbb{R}^{n\times n}$ the extension of $f$ via \eqref{eq:ex}, the preimage
$\textswab{f}^{-1}(\mathcal{S})$ has $n^2$-dimensional Lebesgue measure zero. 
\end{lemma}

\begin{proof}
 The image $\textswab{f}(A)= \sum c_j A^j$ of some $A=(a_{i,j})$ has each entry a scalar power series in its $n^2$ scalar entries, say $P_{i,j}(a_{1,1},\ldots,a_{n,n})\in \mathbb{R}[[X_1,\ldots,X_{n^2}]]$
 for $1\le i,j\le n$. So if $\textswab{f}(A)$ lies in a proper affine subspace 
 $\mathcal{S}$ with
 equation $d_{0}+d_{1,1} x_1+\cdots+d_{n,n}x_{n^2}=0$, $d_0, d_{i,j}\in \mathbb{R}$ not all $0$,
 it satisfies some fixed scalar power series equation
 \[
 Q(a_{1,1},\ldots,a_{n,n})= \sum_{1\le i,j\le n} d_{i,j} P_{i,j}(a_{1,1},\ldots,a_{n,n})=0
 \]
 for $Q(a_{1,1},\ldots,a_{n,n})\in \mathbb{R}[[X_1,\ldots,X_{n^2}]]$ 
 a power series in $n^2$ variables $a_{1,1},\ldots,a_{n,n}$. If $Q\not\equiv 0$
 does not vanish identically, it is
 well-known and for example consequence of Lebesgue Density Theorem that only a set of Lebesgue measure zero in $\mathbb{R}^{n^2}$ can satisfy such an identity, and we are done. So assume $Q\equiv 0$. This means every $A\in \mathbb{R}^{n\times n}$ satisfies $\textswab{f}(A)\in \mathcal{S}$, so $\textswab{f}(U)\subseteq \mathcal{S}$. However, this forces $f$ to be constant as we show. 
 Let $W:=f(I)\subseteq \mathbb{R}$ be
the image of the real function $f: I\to \mathbb{R}$. 
Then, if $f$ is not constant, by its continuity $W$ is an interval with non-empty interior.
Then any matrix
of the form $B=PDP^{-1}$
with $D=\rm{diag}(d_1,\ldots,d_n)\in \mathbb{R}^{n\times n}$ for $d_i\in W$
and $P\in \mathbb{R}^{n\times n}$ invertible is in the image $\textswab{f}(U)$.
Indeed, using \eqref{eq:ex} we have $B=\textswab{f}(PCP^{-1})$ with $C=\rm{diag}( f^{-1}(d_1), \ldots, f^{-1}(d_n))$ where we choose any preimages if they are not unique. However, we claim that the set of obtained matrices $B$ has non-empty interior, contradicting $\textswab{f}(U)\subseteq \mathcal{S}$. Indeed, choose any 
$(d_1,\ldots,d_n)\in W^n$ with pairwise distinct $d_i\ne d_j$ not on the boundary of $W$ (if exists). Then for any induced $B$ it follows from the continuous dependency  
of the eigenvalues from the matrix entries that any matrix $E\in \mathbb{R}^{n\times n}$ within some neighborhood of $B$ in $\mathbb{R}^{n^2}$ still has pairwise distinct real eigenvalues in $W$ (if some were complex they would come in pairs and we would have too many). Thus by the above argument applied to some $\tilde{D}=\rm{diag}(\tilde{d}_1,\ldots,\tilde{d}_n)\in \mathbb{R}^{n\times n}$ for $\tilde{d}_i\in W$ and $\tilde{P}$, 
the matrix $E$ is also contained in the image $\textswab{f}(U)$.
\end{proof}

For $n=1$ the claim follows directly from~\cite{asa}, noticing that any 
non-constant analytic function on an interval $I$ is nowhere constant on $I$ by Identity Theorem (see also~\S~\ref{intro}).
For $n>1$ we reduce it to this case.
Consider $A= \rm{diag}(a, B)$ for $a\in \mathscr{L}_{1,1}$ any Liouville number as above, i.e such that $f_k(a)\in \mathscr{L}_{1,1}$ as well, and $B$ for the moment an arbitrary real $(n-1)\times (n-1)$ matrix.
Since $f_k$ are analytic we have 
\[
\textswab{f}_k(A)= \rm{diag}(f_k(a), \textswab{f}_k(B)),
\]
where by abuse of notation we keep the notation $\textswab{f}_k$ for maps $\textswab{f}_k: \tilde{U}\subseteq \mathbb{R}^{(n-1)\times (n-1)}\to \mathbb{R}^{(n-1)\times (n-1)}$ defined in the same way as in \eqref{eq:ex}.
Hence
the system of inequalities induced by $\textswab{f}_k(A)\textbf{q}-\textbf{p}$ decouples.
So $|q_1 f_k(a)- p_1|< q_1^{-N}$ for some $p_1,q_1, N$ implies $\textbf{p}=(p_1,0,\ldots,0)^{t}$ and $\textbf{q}=(q_1,0,\ldots,0)^{t}$ will induce equally good approximations
$\Vert \textswab{f}_k(A)\textbf{q}-\textbf{p}\Vert=|q_1 f_k(a)-p_1|$ and $\Vert \textbf{q}\Vert=|q_1|$,
so $\Vert \textswab{f}_k(A)\textbf{q}-\textbf{p}\Vert< \Vert \textbf{q}\Vert^{-N}$.
By irrationality of $a$ clearly $\textswab{f}_k(B)$ and $\textswab{f}_k(A)$, $k\ge 0$, satisfying \eqref{eq:tiere} is equivalent, 
with $\textswab{f}_0(C):=C$
the identity map. So
it suffices to show that a full measure set of $B$ also gives rise to $\textswab{f}_k(B)$, $k\ge 0$, satisfying \eqref{eq:tiere}.

Choose $B$ with the property that for any $C_k:=\textswab{f}_k(B)$, $k\ge 0$, there is no non-trivial relation
$C_k\textbf{q}_k-\textbf{p}_k=\textbf{0}$ with $(\textbf{p}_k,\textbf{q}_k)\in \mathbb{Z}^m\times \mathbb{Z}^n$. 
We show that this is possible.
Indeed, any such relation restricts $C_k$ to a proper affine rational subspace $\mathcal{S}_k$ of 
$\mathbb{R}^{(n-1)^2}$, so the exceptional $B=\textswab{f}_k^{-1}(C_k)\subseteq \textswab{f}_k^{-1}(\mathcal{S}_k)$ lie in the preimage of this subspace, thus by Lemma~\ref{depat} applied for dimension $n-1$
and $\mathcal{S}=\mathcal{S}_k$
form a set of Lebesgue measure zero. Since there are only countably many $f_k$ and countably
many relations (affine rational subspaces), a full measure set in the complement remains to choose $B$ from
that give rise to $A$ as in the theorem.


\section{Proof of Theorem~\ref{subring} }

  For (i),
 assume $A\in \mathbb{R}^{n\times n}$ is algebraic over $\mathbb{Z}[X]$. Then $P(A)=\textbf{0}$ for some non-zero $P\in\mathbb{Z}[X]$. Assume $A=B+C$ for some $n\times n$ matrices
 $B,C$. Then also $P(B+C)=\textbf{0}$. However,
 $P(X+Y)$ can be expanded
 in a bivariate (non-commutative) polynomial $Q(X,Y)\in \mathbb{Z}[X,Y]$, so that in particular $P(A)=P(B+C)=Q(B,C)$. Now $Q$ does not induce the zero
 function as otherwise putting $C=\textbf{0}$ we would get that $Q(B,\textbf{0})=P(B)=\textbf{0}$ vanishes for all $B$,
 but since the polynomial ring in one variable is not a polynomial identity ring, this implies $P(X)\equiv \textbf{0}$, against our assumption. This means $B,C$ are not weakly algebraically independent. Taking the contrapositive, if $A=B+C$
 holds for some weakly algebraically independent $B,C$, then $A$ is transcendental.

 For (ii),
 given any real $n\times n$ matrix $A$,
 by Theorem~\ref{gdel} there is a dense $G_{\delta}$ set $\mathscr{H}_A\subseteq \mathscr{L}_{n,n}$ consisting of real $n\times n$ Liouville matrices $B$ so that
$A-B\in  \mathscr{L}_{n,n}$ as well.
 We need to show that if 
 $A$ is not weakly algebraic, then
 for some $B$ as above the matrices $B, A-B$ are
 weakly algebraically independent over $\mathbb{Z}[X,Y]$.

For given $A$ and $P\in \mathbb{Z}[X,Y]$ not inducing the zero function over the matrix ring, denote by $\mathscr{G}_{P,A}$ the set of $B\in \mathbb{R}^{n\times n}$ with $P(B,A-B)=\textbf{0}$. Then $\mathscr{G}_{P,A}$ is closed in $\mathbb{R}^{n\times n}$ by continuity. Assume $\mathscr{G}_{P,A}$ has empty interior for all $P$. Then by countability of $\mathbb{Z}[X,Y]$ the
complement of the union $\cup \mathscr{G}_{P,A}$, taken over all $P\in \mathbb{Z}[X,Y]$
not inducing $\textbf{0}$,
is a dense $G_{\delta}$ set. Hence
as $\mathscr{H}_A$ is also dense $G_{\delta}$, the set
$\mathcal{T}_A:=(\cup \mathscr{G}_{P,A})^c\cap \mathscr{H}_A$
is again dense $G_{\delta}$, in particular non-empty.
Then any $B\in \mathcal{T}_A$ is suitable.

So assume otherwise for some non-zero $P\in \mathbb{Z}[X,Y]$ not inducing the 
$\textbf{0}$ function,
the set $\mathscr{G}_{P,A}$ has non-empty interior. Since every entry of $P(B,A-B)\in \mathbb{R}^{n\times n}$ is a multivariate scalar polynomial in the $n^2$ entries of $B$
it is clear that then $\mathscr{G}_{P,A}=\mathbb{R}^{n\times n}$ is the entire matrix set. However,
this means that there is a relation
$P(B,A-B)=\textbf{0}$ for some $P\in \mathbb{Z}[X,Y]$ not inducing 
the $\textbf{0}$ function, 
and every $B$. 
However, we can expand 
$P(Y,X-Y)$ into a bivariate polynomial in (non-commutative) standard 
form $R(X,Y)$, $R\in \mathbb{Z}[X,Y]$, so that in particular
$P(B,A-B)=R(A,B)$ for any $B$. On the one hand,
$R(A,B)=P(B,A-B)=\textbf{0}$ for all $B$. On the other hand,
$R$ is not inducing the zero function either. Indeed, if so this would mean
$R(C_0, C_1)=P(C_1, C_0-C_1)=\textbf{0}$ for all $C_0, C_1\in \mathbb{R}^{n\times n}$,
implying further that for any $D_0, D_1\in \mathbb{R}^{n\times n}$ if we let $C_1=D_0, C_0=D_0+D_1$ we get $P(D_0, D_1)=P(C_1, C_0-C_1)=R(C_0, C_1)=\textbf{0}$, contradicting 
the assumption that $P$ is not inducing the $\textbf{0}$ function.
Combining these properties, we see that $A$ is weakly algebraic 
(for $\ell=1$),
against our assumption. This argument shows that if $A$ is not weakly algebraic then
there exist weakly algebraically independent $B,C\in \mathscr{L}_{n,n}$ with $B+C=A$.

\section{Proof of Proposition~\ref{pro} (Sketch)}  \label{unten}
We show that the Liouville property
is invariant under the maps $A\to RA$,
$A\to AR$, for 
regular $R\in \mathbb{Q}^{n\times n}$,
$A\to A+T$ for $T\in \mathbb{Q}^{n\times n}$ and $A\to A^{-1}$ for regular $A$. Then clearly it is preserved for the matrices in the proposition that arise from composition of these operations.
For $A\to RA$
it suffices to
modify good approximation pairs  
$(\textbf{p},  \textbf{q})\in \mathbb{Z}^m\times \mathbb{Z}^n$ 
with respect to $A$ via $(\textbf{p}^{\prime},  \textbf{q}^{\prime}):=(MN\cdot R\textbf{p}, MN\cdot \textbf{q})$, for $M,N\in \mathbb{Z}$ 
the common denominators of the rational entries of $R,S$. Similarly for $A\to AR$
we take $(\textbf{p}^{\prime},  \textbf{q}^{\prime}):=(M^{\prime}\cdot \textbf{p}, M^{\prime}\cdot R^{-1} \textbf{q})$ for $M^{\prime}\in \mathbb{Z}$
the common denominator of the rational matrix $R^{-1}$.
For $A\to A+T$ with a rational matrix $T$,
write $T=T^{\prime}/N^{\prime}$ with an integer
matrix $T^{\prime}$ and an integer $N^{\prime}$. Then we take $(\textbf{p}^{\prime},  \textbf{q}^{\prime}):=(N^{\prime}\cdot T^{\prime}\textbf{p}+ N^{\prime}\textbf{p}, N^{\prime}\cdot \textbf{q})$.
For $A\to A^{-1}$ we take the reversed
vector $(\textbf{p}^{\prime},  \textbf{q}^{\prime}):=(\textbf{q}, \textbf{p})$.
Moreover, using $R_i$ are invertible,
it is clear that \eqref{eq:tiere} is preserved for the obtained matrices.
We leave the details to the reader. 

\section{Final remarks on  Theorem~\ref{1} } \label{s4}

We believe the metrical bound in Theorem~\ref{1} can be improved at least to $(n-2)^2+2$. As the full set $\mathscr{L}_{n,n}$ has Hausdorff dimension $n(n-1)$, see~\cite{bv}
for a considerably 
more general claim,
in particular this would be optimal for $n=2$. We sketch the proof of a special case.
Let us restrict to finitely many polynomials $f_k(z)= c_{0,k}+\cdots+c_{J,k} z^J$, $1\le k\le K$,
of degrees $J=J(k)$ at least two. 
For $n=2$, to obtain this result we consider the larger class of matrices 
\[
	A=\begin{pmatrix}
	a & b \\
	c & 0
	\end{pmatrix}
	\] 
	where again $b\in \mathscr{L}_{1,1}$ is fixed but $a,c$ are real parameters. Again $A\in \mathscr{L}_{2,2}$ is easily seen as soon as $a$ avoids some countable set. The powers of such matrices will have the form
 	\[
	A^j= \begin{pmatrix}
	v_j & w_j \\
	\ast & \ast
	\end{pmatrix}, \qquad j\ge 1,
	\] 
 with $v_j, w_j$ polynomials in $a,c$
 satisfying the recursions
 \[
	v_j= a r_{j-1}+bc v_{j-2},\qquad   w_j= b v_{j-1}.
	\]
 Now it can be shown with Inverse Function
 Theorem in place of the deep result
 from~\cite{km}, that on a joint non-empty open set $V\subseteq \mathbb{R}^2$ the maps $\theta_k: V\to \mathbb{R}^2$, $1\le k\le K$, defined by
 \[
 \theta_k(a,c)=\left(\sum_{j=0}^{J} c_{j,k} v_j, \sum_{j=0}^{J} c_{j,k} w_j\right)\in\mathbb{R}^2
 \]
 with image the first line of $\textswab{f}_k(A)$, induce local diffeomorphisms onto their open images $\theta_k(V)=:V_k$
 in $\mathbb{R}^2$.
 By a Khintchine type result
 the images contain large (full 
 measure within the total image $V_k$) subsets $\mathscr{S}_k\subseteq V_k\subseteq \mathbb{R}^2$
 of line vectors $\textbf{z}_{k}$ with exponent $\omega^{1\times 2}(\textbf{z}_{k})=2$.
 Intersecting finitely many  preimages $\mathscr{S}:= \cap \theta_k^{-1}(\mathscr{S}_k)$ over $1\le k\le K$, by the locally Lipschitz property of $\theta_k^{-1}$, we still get
 a full measure set
 $\mathscr{S}\subseteq V\subseteq \mathbb{R}^2$ within $V$. For any pair $(a,c)\in \mathscr{S}$ with corresponding matrix $A$, 
 as $\textbf{z}_k$ is the first row of $\textswab{f}_k(A)$, by Proposition~\ref{hauser} we have 
 $\textswab{f}_k(A)\notin \mathscr{L}_{2,2}$, in fact
 \[
 \omega^{2\times 2}(\textswab{f}_k(A))\le \omega^{1\times 2}(\textbf{z}_k)= 2,\qquad 1\le k\le K.
 \]
This finishes the proof for $n=2$. The extension to larger $n$ works by considering diagonal blocks analogously to the proof
 of Theorem~\ref{1}. This method also allows us to show that regular matrices $A$ satisfy
 this relaxed version of 
 Theorem~\ref{1}.
 
 Presumably the case of countably many analytic functions not of the form \eqref{eq:four},
 as in Theorem~\ref{1}, instead of finitely many polynomials,
 can be treated with some refined argument.
 The finiteness is only used in the above argument to guarantee
 a joint open set for all functions $f_k$ simultaneously to apply Inverse Function Theorem on; as for transitioning to (almost) arbitrary
 analytic functions, it may be intricate
 to verify the regularity 
 hypothesis of the Inverse Function Theorem
 in this general setup. It may happen that some more analytic functions have to be excluded to achieve this.

\end{document}